\documentclass[a4paper,12pt,reqno,dvipsnames]{amsart}
\usepackage{xcolor}
\usepackage{graphicx}
\usepackage{amsmath}
\usepackage{amssymb}

\allowdisplaybreaks

\theoremstyle{plain}
\newtheorem{theorem}{Theorem}[section]
\newtheorem{proposition}[theorem]{Proposition}
\newtheorem{lemma}[theorem]{Lemma}

\newtheorem{remark}[theorem]{Remark}
\newtheorem{question}[theorem]{Question}

\newtheorem{maintheorem}{Theorem} 

\newtheorem{maincorollary}{Corollary}  

\makeatletter
\let\c@maincorollary\c@maintheorem % Share the theorem counter
\makeatother

\theoremstyle{definition} 
\newtheorem*{ack}{Acknowledgements}

\newcommand{\N}{\mathbb{N}}

\newcommand{\R}{\mathbb{R}}

\let\altphi\phi
\let\phi\varphi
\let\varphi\altphi
\let\altphi\undefined

\newcommand{\Ric}{{\rm Ric}}

\newcommand{\II}{\mathrm{I\!I}}

\usepackage{color}
\usepackage{hyperref}
\hypersetup{
	colorlinks=true,
	linkcolor=blue,
	citecolor=blue,
	urlcolor=blue
}

\author{Davi Maximo}
\address{Department of Mathematics, University of Pennsylvania}

\author{Philipp Reiser}
\address{Department of Mathematics, University of Fribourg}

\author{Daniele Semola}
\address{Faculty of Mathematics, University of Vienna}

\begin{document}
	
	\title[Ricci curvature, minimal hypersurfaces, Betti numbers]{Ricci curvature and minimal hypersurfaces 
    with large Betti numbers}

	\maketitle

    \begin{abstract}
     {%\color{blue}
     In any dimension $n+1\ge 4$ we construct a sequence of closed $(n+1)$-dimensional Riemannian manifolds with positive Ricci curvature admitting embedded two-sided minimal hypersurfaces such that the following hold: 
     \begin{itemize}
     \item[(i)] any such hypersurface has Morse index one;
     \item[(ii)] the first Betti numbers of the hypsersurfaces are not uniformly bounded along the sequence.
     \end{itemize}}
    \end{abstract}

\section{Introduction}

Recent years have seen significant advances in the existence theory of minimal hypersurfaces in higher-dimensional Riemannian manifolds, largely propelled by the influential work of F.C. Marques and A. Neves \cite{MaNe14, MaNe16, MaNe17, MaNe21}. Their contributions refined and extended the min-max framework originally developed by F. Almgren and J. Pitts in the 1980s, and have inspired a new wave of results in the field. Notably, the works X. Zhou \cite{Zh20}, A. Song \cite{Son23}, and others \cite{ChMan20,IMN,LioMaNe} have further built upon this foundation.

A key feature of this theory is its reliance on weak convergence in the varifold setting, which ensures the existence of a limiting minimal hypersurface but does not provide precise control over its topology -- leaving open questions about the specific geometric structure of the resulting surface. However, a remarkable strength of these modern refinements is their ability to predict the Morse index of the minimal hypersurface, offering insight into its stability and the number of independent directions in which it can be deformed to reduce area. 

 This interplay between existence, index prediction, and topological ambiguity underscores the power of the theory and its ongoing evolution in geometric analysis. On several instances, estimates that show the Morse index controlling the topology of the hypersurface can be established, $e.g.$~\cite{ACS18,ChKeMa17,ChMa23,  MR20, MaxVP,So23}. However, as the ambient geometry becomes more complex, the interplay between index and topology becomes subtler, and many fundamental questions remain open. 
 
 One guiding question in this direction is the following folklore conjecture: if $\Sigma^{n}$ is a closed, two-sided, minimal hypersurface of Morse index 1 on a {closed} manifold $(M^{n+1},g)$ of positive Ricci curvature, then the first Betti number $b_1(\Sigma^{n})$ must be bounded by a universal constant $c(n)$, {see \cite[Section 8]{NevesICM}}. This is known to be the case when $n+1=3$ \cite{Yau87}. The main result of this paper is the construction of counterexamples to this conjecture starting in dimension $n+1=4$: 

\begin{maintheorem}\label{mainthm2}
   For $n+1\geq 4$, there exists a sequence of smooth, {closed}, $(n+1)$-dimensional, Riemannian manifolds $(M^{n+1}_k,g_k)$ with embedded, two-sided minimal hypersurfaces $\Sigma_k\subset M_k$ satisfying the following properties:
\begin{itemize}
\item[i)] $\Ric_{g_k}> 0$ for any $k\in\N$;
\item[ii)] $\mathrm{index}(\Sigma_k)=1$ for any $k\in\N$;
\item[iii)] $b_1(\Sigma_k)\to \infty$ as $k\to\infty$, where $b_1$ stands for the first Betti number.
\end{itemize}
\end{maintheorem}

To describe our construction in detail, we recall that if $X^n$ is a closed manifold, the {\it suspension} $\Sigma_0X$ of $X$ {%\color{ForestGreen} 
(alternatively also called the \emph{spinning} of $X$)} is a $(n+1)$-dimensional manifold obtained by a surgery on $X\times S^1$ as follows: let $D^n\subset X$ be an embedded disk and define (see \cite{Du24,GR25,Re24}): 
\begin{equation*}
\Sigma_0X = \left((X\setminus D^n )\times S^1 \right) \cup_{S^{n-1}\times S^1} \left(S^{n-1}\times D^2\right)\, .
\end{equation*}
An interesting feature of $\Sigma_0(\cdot)$ is that it preserves various topological properties, such as the fundamental group (see \cite[Lemma 5.2]{GR25}). 

Theorem \ref{mainthm2} is a direct consequence of the following result:
 
\begin{maintheorem}\label{mainthm}
Let $M^n$, $n\geq 3$, be a closed manifold that admits a Riemannian metric of positive Ricci curvature. Then, for all $\ell\geq 0$, there exists a Riemannian metric of positive Ricci curvature on
        \begin{equation*} 
        \Sigma_0 M\#_\ell (S^2\times S^{n-1}) 
        \end{equation*}
         and a totally geodesic two-sided embedding
         \begin{equation*} 
         M\#(-M)\#_\ell(S^1\times S^{n-1})\subseteq \Sigma_0 M\#_\ell (S^2\times S^{n-1}) 
         \end{equation*}
         of index one.
\end{maintheorem}

In particular, we have the following Corollary for dimension $n+1=4$:

\begin{maincorollary}\label{maincor}
        Let $N=S^3/\Gamma$ be a spherical space form, i.e.\ $\Gamma\subseteq O(4)$ is a finite subgroup that acts freely on $S^3$. Then, for any $\ell\geq0$, the manifold
        \begin{equation*} 
        N\#(-N)\#_\ell(S^1\times S^2) 
        \end{equation*}
        can be realised as a two-sided minimal hypersurface of index one in a closed Riemannian 4-manifold of positive Ricci curvature.
\end{maincorollary}

The metric in Theorem \ref{mainthm} is constructed as follows: The starting point is the construction of J.-P. Sha and D.-G. Yang \cite{SY91}, which provides a metric of positive Ricci curvature on the space obtained from $M\times S^1$ by performing $(\ell+1)$-surgeries along the $S^1$-factor, resulting in the manifold $\Sigma_0 M\#_\ell(S^2\times S^{n-1})$. We then show that a reflection on the $S^1$-factor of $M\times S^1$ induces an isometric action on this space with fixed point set $M\#(-M)\#_\ell (S^1\times S^{n-1})$, which is a totally geodesic hypersurface, such that the induced metric has positive scalar curvature. Finally, we establish a deformation result (Proposition \ref{P:deformation}) that shows that under these hypotheses the metric on the ambient manifold can be deformed so that the hypersurface is minimal of index one while keeping the ambient Ricci curvature positive.

We note that the radius $r$ of the circle factor in the initial manifold $M\times S^1$ can be chosen arbitrarily small. As a consequence, since varying $r$ only affects the hypersurface in the region where the surgeries are performed, its area has a uniform positive lower bound for all $r$. The same holds for the diameter of the ambient manifold, while the volume of the ambient manifold converges to $0$ as $r\to 0$.

\medskip

%\begin{remark}
Following Chodosh-Li-Stryker in  \cite[Remark B.3]{ChodoshStrykerLi}, the deformation theorem in Proposition \ref{P:deformation} can also be used to prove the following:

\begin{maintheorem}
There exists a closed Riemannian $4$-manifold $(X^4, g)$ with $\Ric > 0$ that admits a complete, two-sided, stable minimal hypersurface immersed in it. 
\end{maintheorem}
%\end{remark} 

In view of Corollary \ref{maincor}, the following question arises naturally:
\begin{question}\label{Q:3-manifolds}
    Let $N$ be closed, oriented $3$-manifold that admits a Riemannian metric of positive scalar curvature. Can $N$ be realised as a minimal hypersurface of index one within a closed Riemannian 4-manifold of positive Ricci curvature?
\end{question}
By the work of R. Schoen and S.-T. Yau \cite{ScYa79} and G. Perelman \cite{perelman2002entropy,perelman2003finite,perelman2003ricci}, see also \cite{Ma12}, a closed, oriented $3$-manifold $N$ admits a Riemannian metric of positive scalar curvature if and only if it is the connected sum of finitely many spherical space forms and copies of $S^1\times S^2$. Besides the topologies arising from the application Corollary~\ref{maincor}, one can also realise spherical space forms as index one minimal hypersurfaces within closed 4-manifolds of positive Ricci curvature. {%\color{blue}
This follows from \cite{zhou2024}, where the author constructs complete Riemannian manifolds with nonnegative Ricci curvature which are isometric to a cone over $S^3/\Gamma$ outside a compact set for any spherical space form $S^3/\Gamma$, combined with a standard doubling and smoothing argument based on Perelman's gluing lemma from \cite{Perelmanbetti}.}

\medskip

This paper is organised as follows: In Section \ref{S:prel} we establish a deformation result to obtain hypersurfaces of index one. In Section \ref{S:proof} we then recall the construction of Sha--Yang and use the deformation result to prove Theorem \ref{mainthm}. Finally, in Appendix \ref{appendix}, we detail the proof of the equivariant deformation result for Ricci curvature, which is used implicitly in the main argument. Although the non-equivariant version goes back to Ehrlich, the extension to group actions involves a subtle averaging issue that had not been fully addressed in the literature.  

\begin{ack}
    The authors thank the referee for helpful comments and suggestions. The second named author would like to thank the Faculty of Mathematics of the University of Vienna for its hospitality while parts of this work were carried out.
\end{ack}

 \section{Preliminaries}\label{S:prel}
In this section, we will gather some preparatory results about metrics of positive Ricci curvature with a totally geodesic hypersurface.

Recall that for a Riemannian metric on a cylinder $(-\delta,\delta)\times N$ of the form $dt^2+g_t$, where $N$ is a manifold and $g_t$ a smoothly varying family of Riemannian metrics on $N$, the second fundamental form of the hypersurface $\{t\}\times N$ with respect to the unit normal $\partial_t$ is given by
\begin{equation} 
\II=-\frac{1}{2}g_t'\, . 
\end{equation}
Here $g_t'$ denotes the $t$-derivative of $g_t$, i.e.\ for $u,v\in T_xN$ we have
\begin{equation}
g_t'(u,v)=\frac{\partial}{\partial t}g_t(u,v)\, . 
\end{equation}
The expression for $\II$ e.g.\ follows from \cite[Section 3.2.1]{Pe16}, where we use that the Hessian $\mathrm{Hess}f$ of the function $f(t,x)=t$ on $(-\delta,\delta)\times N$ is given by $\mathrm{Hess}f=-\frac{1}{2}g_t'$.

In particular, the hypersurface $\{t\}\times N$ is totally geodesic if and only if $g_t'=0$.

	\begin{lemma}\label{L:metric_tot_geod}
		Let $(N^{n-1},g_N)$ be a closed Riemannian manifold. Then the following statements are equivalent:
		\begin{enumerate}
			\item[(i)] For all $\varepsilon>0$ there exists $\delta>0$ and a metric $g^\varepsilon=dt^2+g_t^\varepsilon$ on $(-\delta,\delta)\times N$ of positive Ricci curvature such that $\{0\}\times N$ is totally geodesic, $g_0^\varepsilon=g_N$, and $\Ric(\partial_t,\partial_t)=\varepsilon$ at $t=0$.
			\item[(ii)] $(N,g_N)$ has non-negative scalar curvature.
		\end{enumerate}
	\end{lemma}
	\begin{proof}
		For any metric $g$ of the form $dt^2+g_t$ on $(-\delta,\delta)\times N$ such that $\{0\}\times N$ is totally geodesic the Ricci curvatures at $t=0$ are given as follows:
		\begin{align}
			\Ric(\partial_t,\partial_t)&=-\frac{1}{2}\mathrm{tr}_{g_0}g_0''\, ,\\
			\Ric(v,\partial_t)&=0\, ,\\
			\Ric(u,v)&=\Ric^{g_0}(u,v)-\frac{1}{2}g_t''(u,v)
		\end{align}
		for $u,v\in TN$, see e.g.\ \cite[Section 3.2.1]{Pe16}.
		
		Now suppose $\varepsilon>0$ and $g=g^\varepsilon$ as in (i). Then it follows that
		\begin{equation} 
        \varepsilon=-\frac{1}{2}\mathrm{tr}_{g_0}g_0''=\sum_{i=1}^{n-1}\left(\Ric^g(e_i,e_i)-\Ric^{g_0}(e_i,e_i)\right)>-\mathrm{scal}^{g_0}\, , 
        \end{equation}
		where $(e_i)$ is a local orthonormal frame of $(N,g_N)$. Since $g_0=g_N$, it follows that $\mathrm{scal}^{g_N}>-\varepsilon$ for all $\varepsilon$, i.e.\ $\mathrm{scal}^{g_N}\geq 0$.
		
		Conversely, assume that $\mathrm{scal}^{g_N}\geq 0$ and let $\varepsilon>0$. For $\delta>0$ we then define the metric $g^\varepsilon=dt^2+g_t^\varepsilon$ on $(-\delta,\delta)\times N$, where
		\begin{equation} 
        g_t^\varepsilon=g_N+t^2h 
        \end{equation}
		and
		\begin{equation} 
        h=\Ric^{g_N}-\frac{\mathrm{scal}^{g_N}+\varepsilon}{n-1}g_N\, .
        \end{equation}
		Note that $g_t^\varepsilon$ is a Riemannian metric for $\delta$ (and hence $|t|$) sufficiently small. We then have $g_0^\varepsilon=g_N$, ${g_0^\varepsilon}'=0$, and ${g_0^\varepsilon}''=2h$. In particular, $\{0\}\times N$ is totally geodesic. Further, the Ricci curvatures at $t=0$ are given by
		\begin{align}
			\Ric(\partial_t,\partial_t)&=-\mathrm{tr}_{g_N}h=\varepsilon\, ,\\
			\Ric(v,v)&=\Ric^{g_N}(v,v)-h=\frac{\mathrm{scal}^{g_N}+\varepsilon}{n-1}g_N(v,v)>0
		\end{align}
		and $\Ric(\partial_t,v)=0$ for every $v\in TN\setminus\{0\}$. In particular, the metric $g^\varepsilon$ has positive Ricci curvature at $t=0$ and therefore also on $(-\delta,\delta)\times N$ if we choose $\delta$ sufficiently small.  This shows that $g^\varepsilon$ is a metric as in (i).
	\end{proof}

   \begin{remark}
 An observation similar to the one employed in the proof of the implication from (ii) to (i) in Lemma \ref{L:metric_tot_geod} appears in \cite[Appendix B.2]{ChodoshStrykerLi}.       
   \end{remark}

Using Lemma \ref{L:metric_tot_geod}, we prove the main result of this section:

    \begin{proposition}\label{P:deformation}
		Let $(M^n,g_0)$ be a closed Riemannian manifold of positive Ricci curvature and let $N^{n-1}\subseteq M$ be an embedded totally geodesic two-sided hypersurface. Suppose that $g_0$ induces a metric of non-negative scalar curvature on $N$. Then, for any $\varepsilon>0$, there exists a deformation $g_t$, $t\in[0,1]$, of $g_0$ with the following properties:
		\begin{enumerate}
			\item[(i)] The metric $g_t$ has positive Ricci curvature for all $t$;
			\item[(ii)] The deformation $g_t$ is constant outside an arbitrarily small neighbourhood of $N$;
			\item[(iii)] The induced metric on $N$ remains unchanged and $N$ is totally geodesic for all $g_t$;
			\item[(iv)] The normal Ricci curvatures of $g_1$ at $N$ satisfy $\Ric^{g_1}(\nu_N,\nu_N)=\varepsilon$, where $\nu_N$ denotes a unit normal vector field at $N$.
		\end{enumerate}
	\end{proposition}

	\begin{proof} 
		By considering normal coordinates around $N$, since $N\subseteq M$ is two-sided, we can identify a neighbourhood of $N$ with $(-\delta,\delta)\times N$ for some $\delta>0$ and write the metric $g_0$ in this neighbourhood as $g_0=ds^2+h_s$, where $h_s$ is a family of metrics on $N$. The condition of $N$ being totally geodesic is then equivalent to $h_s'=0$ at $s=0$.
		
		Hence, for $g_N=g_0|_N$, the metrics $g^\varepsilon$ from Lemma \ref{L:metric_tot_geod} and $g_0$ coincide up to first order at $\{0\}\times N$. The existence of the required deformation therefore follows from \cite[Theorem 1.10]{Wr02}, see also \cite{BH22}.
	\end{proof}

\begin{remark}
    With the help of Proposition \ref{P:deformation} one can answer in the affirmative the question raised in \cite[Remark B.3]{ChodoshStrykerLi}.
\end{remark}

	\section{Proof of Theorem \ref{mainthm}}\label{S:proof}

    We start with the construction of Sha--Yang \cite{SY91} to construct a metric of positive Ricci curvature on the connected sum $\Sigma_0 M\#_\ell(S^2\times S^{n-1})$, where $M$ admits a metric of positive Ricci curvature $g$. We address the reader also to \cite{AndersonDuke} for a different approach to similar constructions.
    
    For the construction we assume that there are $(\ell+1)$ pairwise disjoint embeddings of discs
    \begin{equation} \bigsqcup_{\ell+1} D^{n}_{r_2}\subseteq (M,g)\, , 
    \end{equation}
    each equipped with the induced metric of a geodesic ball of some radius $r_2>0$ in the round sphere $S^{n}$ of radius $1$. This can always be achieved by local deformations around $(\ell+1)$ points while preserving positive Ricci curvature, see \cite[Theorem 1.10]{Wr02}, \cite{BH22} or \cite[Lemma 4.2]{Re24}.

    We now consider the product metric $r_1^2ds_1^2+g$ on $S^1\times M$. Note that this metric has non-negative Ricci curvature. The connected sum $\Sigma_0 M\#_\ell (S^2\times S^{n-1})$ is then obtained by performing surgery along the corresponding embeddings $\bigsqcup_{\ell+1}(S^1\times D_{r_2}^{n})\subseteq (S^1\times M)$, that is, we remove the interior of each copy of $S^1\times D^{n}_{r_2}$ in $S^1\times M$ and glue in a copy of $D^2\times S^{n-1}$ along the resulting boundary component $S^1\times S^{n-1}$. For a proof that this space results in $\Sigma_0 M\#_\ell(S^2\times S^{n-1})$ as claimed see e.g.\ \cite[Lemma 2.7]{Re24}.
	
	From a metric perspective, the approach taken in \cite{SY91} consists of considering normal coordinates around the centre of each $D_{r_2}^n$ and viewing the metric on each embedded $S^1\times D_{r_2}^{n}$ as a doubly warped product metric with constant warping function for the $S^1$-factor, i.e.
	\begin{equation} 
    r_1^2ds_1^2+(dt^2+\sin^2(t)ds_{n-1}^2)=dt^2+r_1^2ds_1^2+\sin^2(t)ds_{n-1}^2 
    \end{equation}
	for $t\in[0,r_2]$. To perform the surgery, this metric now gets replaced by a different doubly warped product metric
	\begin{equation}\label{EQ:doubly_warped}
		dt^2+h(t)^2ds_1^2+f(t)^2ds_{n-1}^2
	\end{equation}
	on $[t_0,r_2]\times S^1\times S^{n-1}$ for some $t_0<r_2$ and smooth warping functions $h,f\colon[t_0,r_2]\to[0,\infty)$ that coincide with the previous warping functions near $t=r_2$, i.e.:\
	\begin{enumerate}
		\item[(i)] $h(r_2)=r_1$ and all derivatives of $h$ vanish at $t=r_2$;
		\item[(ii)] $f(t)=\sin(t)$ near $t=r_2$.
	\end{enumerate}
	
	To achieve the change in the topology, i.e.\ to replace $S^1\times D^{n}$ by $D^2\times S^{n-1}$, the functions $h$ and $f$ satisfy different boundary conditions at $t=t_0$. More precisely, they satisfy the following:
	\begin{enumerate}
		\item[(iii)] The function $h$ is an odd function at $t=t_0$ with $h'(t_0)=1$ (in particular, $h(t_0)=0$);
		\item[(iv)] The function $f$ is an even function at $t=t_0$ with $f(t_0)>0$ (in particular, $f'(t_0)=0$).
	\end{enumerate}
	With these boundary conditions satisfied, we obtain a smooth metric on $D^2\times S^{n-1}$, see e.g.\ \cite[Proposition 1.4.7]{Pe16}. For a construction of such functions that results in a metric of positive Ricci curvature, we refer to \cite[Lemma 1]{SY91}, \cite[Lemma 3.3]{Re24} or Lemma \ref{L:warping_fcts} below.
	
	Next, we construct a totally geodesic embedding $M\#(-M)\#_\ell(S^1\times S^{n-1})\subseteq\Sigma_0 M\#_\ell(S^2\times S^{n-1})$. For that, we first note that the involution on $S^1\times M$ that is a reflection on $S^1$ and the identity on $M$ is isometric with respect to the product metric $r_1^2ds_1^2+g$ with fixed point set $S^0\times M=M\sqcup (-M)$.
	
	By definition, each embedded copy of $S^1\times D_{r_2}^{n}\subseteq S^1\times M$ is preserved by the involution, which, in normal coordinates on $D_{r_2}^{n}$ is given by a reflection of the $S^1$-factor in $[0,r_2]\times S^1\times S^{n-1}$ and the identity on the remaining factors.
	
	Since this involution is isometric for any metric of the form \eqref{EQ:doubly_warped}, it also defines an isometric involution after preforming the surgery, i.e.\ on $\Sigma_0 M\#_{\ell}(S^2\times S^{n-1})$. Its fixed point set, which is necessarily totally geodesic, is the space obtained from $M\times S^0$ by removing $\bigsqcup_{\ell+1}D_{r_2}^n\times S^0$ and gluing in $(\ell+1)$ copies of $D^1\times S^{n-1}$. In other words, it is the manifold obtained by doubling $M\setminus\bigsqcup_{\ell+1}D^n$ along its boundary, which is diffeomorphic to $M\#(-M)\#_\ell (S^1\times S^{n-1})$.
    
    To apply Proposition \ref{P:deformation} it remains to verify that the induced metric has positive scalar curvature.
	
	\begin{lemma}\label{L:warping_fcts}
		For $r_1>0$ sufficiently small there exists a choice of warping functions $h,f$ such that they satisfy the boundary conditions (i)--(iv), the metric \eqref{EQ:doubly_warped} has positive Ricci curvature, and such that the induced metric on $M\#(-M)\#_\ell (S^1\times S^{n-1})\subseteq\Sigma_0 M\#_\ell(S^2\times S^{n-1})$ has positive scalar curvature.
	\end{lemma}
	\begin{proof}
		We follow the construction of \cite[Lemma 3.3]{Re24}, which is a slight modification of \cite[Lemma 1]{SY91}.
        
        The initial step consists of defining $f\colon[0,\infty)\to(0,\infty)$ as the solution of the initial value problem
		\begin{align}
			f(0)&=1\, ,\\
			f'(0)&=0\, ,\\
			f''&=\frac{\alpha\lambda_0^2}{2}f^{-\alpha-1}\, ,
		\end{align}
		where $\lambda_0\in(\cos(r_2),1)$ and $\alpha\in(n-2,\frac{n-2}{\lambda_0^2})$. In particular, $f''>0$ and hence $f\geq 1$. In addition, by integrating the equation $f''f'=\frac{\alpha\lambda_0^2}{2}f^{-\alpha-1}f'$, we obtain
        \begin{equation}\label{EQ:f'}
            {f'}^2=\lambda_0^2(1-f^{-\alpha}).
        \end{equation}
        
        The function $h$ is then defined by 
		\begin{equation} 
        h=\frac{2}{\alpha\lambda_0^2}f'\, . 
        \end{equation}
        By the definition of $f$, we obtain the following.
        \begin{align}
            h'&=f^{-\alpha-1},\\
            h''&=-(\alpha+1)f^{-\alpha-2}f'.
        \end{align}
        In particular, $h'(0)=1$. Moreover, one can show inductively that at $t=0$, the function $f$ is an even function and the function $h$ is an odd function, and hence these functions satisfy the boundary conditions (iii) and (iv) with $t_0=0$.

        To show that the resulting metric \eqref{EQ:doubly_warped} has positive Ricci curvature, we first calculate
        \begin{align}
            \frac{f''}{f}&=\frac{\alpha\lambda_0^2}{2}f^{-\alpha-2},\label{EQ:f''/f}\\
            \frac{h''}{h}&=-\frac{\alpha(\alpha+1)}{2}\lambda_0^2f^{-\alpha-2},\\
            \frac{1-{f'}^2}{f^2}&=\frac{1-\lambda_0^2+\lambda_0^2f^{-\alpha}}{f^2},\label{EQ:(1-f'^2)/f^2}\\
            \frac{f'h'}{fh}&=\frac{\alpha\lambda_0^2}{2}f^{-\alpha-2}
        \end{align}
        Let $\partial_s\in TS^1$ denote a unit tangent vector of $(S^1,ds_1^2)$ and $v\in TS^{n-1}$ a unit tangent vector of $(S^{n-1},ds_{n-1}^2)$. Then we obtain the following for the Ricci curvatures of the doubly warped product metric \eqref{EQ:doubly_warped} (see e.g.\ \cite[Section 4.2.4]{Pe16} for the formulae for the Ricci curvature).
        \begin{align}
            \Ric(\partial_t,\partial_t)&=-\frac{h''}{h}-(n-1)\frac{f''}{f}\notag\\
            &=\frac{\alpha}{2}\lambda_0^2\left(\alpha+1-(n-1)\right)f^{-\alpha-2}\notag\\
            &>\frac{\alpha}{2}\lambda_0^2\left(n-1-(n-1)\right)f^{-\alpha-2}\notag\\
            &=0,\label{EQ:Ric(t)}\\
            \Ric(\tfrac{\partial_s}{h},\tfrac{\partial_s}{h})&=-\frac{h''}{h}-(n-1)\frac{f'h'}{fh}\notag\\
            &=-\frac{h''}{h}-(n-1)\frac{f''}{f}\notag\\
            &=\Ric(\partial_t,\partial_t)\notag\\
            &>0,\label{EQ:Ric(s)}\\
            \Ric(\tfrac{v}{f},\tfrac{v}{f})&=-\frac{f''}{f}+(n-2)\frac{1-{f'}^2}{f^2}-\frac{f'h'}{fh}\notag\\
            &=\frac{-\alpha\lambda_0^2f^{-\alpha}+(n-2)(1-\lambda_0^2+\lambda_0^2f^{-\alpha}))}{f^2}\notag\\
            &\geq \frac{-\alpha\lambda_0^2+(n-2)}{f^2}\notag\\
            &>0,\label{EQ:Ric(v)}\\
            \Ric(\partial_t,\tfrac{\partial_s}{h})&=\Ric(\partial_t,\tfrac{v}{f})=\Ric(\tfrac{\partial_s}{h},\tfrac{v}{f})=0,
        \end{align}
        where, for \eqref{EQ:Ric(t)}, we used $\alpha>n-2$, and for \eqref{EQ:Ric(v)}, we used $f\geq 1$ and $\alpha<\frac{n-2}{\lambda_0^2}$. It follows that the metric \eqref{EQ:doubly_warped} has positive Ricci curvature.
        
        %and that the resulting metric \eqref{EQ:doubly_warped} has positive Ricci curvature. In addition, $f$ satisfies $f''>0$, $f\ge 1$, and $f'(t)\to \lambda_0$ as $t\to\infty$, and $h$ satisfies $h''<0$ and $h'>0$.
		
		We now analyse the scalar curvatures of the fixed point set of the involution on each copy of $D^2\times S^{n-1}$, i.e.\ on $D^1\times S^{n-1}$. Here, the metric is given by
		\begin{equation} 
        dt^2+h(t)^2ds_0^2+f(t)^2ds_{n-1}^2=dt^2+f(t)^2ds_{n-1}^2\, , 
        \end{equation}
        where we extended the interval $[0,\infty)$ to $\R$ and defined $f(-t)=f(t)$ for $t<0$.
        
		Recall that the scalar curvature of such a warped product metric is given by
        \begin{equation}\label{EQ:scal}
        \mathrm{scal}=-2(n-1)\frac{f''}{f}+(n-1)(n-2)\frac{1-{f'}^2}{f^2}\, ,
        \end{equation}
		see e.g.\ \cite[Section 4.2.3]{Pe16}.
        %As shown in the proof of \cite[Lemma 3.3]{Re24}, we have the following:
		% \begin{align}
		% 	\frac{f''}{f}&=\frac{\alpha\lambda_0^2}{2}f^{-\alpha-2}\, ,\\
		% 	\frac{1-{f'}^2}{f^2}&=\frac{1-\lambda_0^2+\lambda_0^2f^{-\alpha}}{f^2}\, .
		% \end{align}
		Hence, from \eqref{EQ:f''/f} and \eqref{EQ:(1-f'^2)/f^2}, we obtain, by using $f\geq 1$ and $\alpha\lambda_0^2<n-2$,
        \begin{align}
         \nonumber   \mathrm{scal}&=(n-1)f^{-\alpha-2}\left( -\alpha\lambda_0^2+(n-2)(\lambda_0^2+(1-\lambda_0^2)f^\alpha )\right)\\
      \nonumber      &\geq (n-1)f^{-\alpha-2}\left( -\alpha\lambda_0^2+(n-2) \right)\\
            &>0\, .
        \end{align}

        It remains to satisfy the boundary conditions (i) and (ii). We first note that, since $f'(0)=0$ and $f''>0$, we have $f(t)\to\infty$ as $t\to\infty$. Hence, by \eqref{EQ:f'}, $f'(t)\to\lambda_0$ as $t\to\infty$. Since $\lambda_0>\cos(r_2)$, there exists $t_1>0$ with $f'(t_1)>\cos(r_2)$. We now modify the functions $f$ and $h$ near $t_1$ to satisfy the boundary conditions (i) and (ii).
        
        Here we again proceed as in \cite[Lemma 3.3]{Re24} and first replace $f$ by $\bar{f}(t)=N\sin(\frac{t-t'}{N})$ for $t\in [t_1,t_1']$, where $N,t'>0$ are chosen so that $f$ and $\bar{f}$ coincide up to first order at $t=t_1$, and $t_1'$ is chosen as the smallest value $t_1'>t_1$ with $\bar{f}'(t_1')=\cos(r_2)$. Since $\bar{f}''(t)<0<f''(t)$ for $t\in[t_1,t_1']$, it follows that
        \begin{equation}
            -\frac{\bar{f}''}{\bar{f}}\geq -\frac{f''}{f},\quad -\frac{\bar{f}'}{\bar{f}}\geq-\frac{f'}{f}\quad\text{and}\quad \frac{1}{\bar{f}^2}\geq\frac{1}{f^2},
        \end{equation}
        so all Ricci curvatures \eqref{EQ:Ric(t)}--\eqref{EQ:Ric(v)} of \eqref{EQ:doubly_warped} and the scalar curvature \eqref{EQ:scal} increase and are in particular still positive when we replace $f$ by $\bar{f}$ for $t\in[t_1,t_1']$. Since they depend linearly on the second derivative of $f$, smoothing $f$ in a small neighbourhood of $t_1$, as e.g.\ in \cite[Lemma 3.1]{RW23}, results in a smooth function for which the Ricci curvatures \eqref{EQ:Ric(t)}--\eqref{EQ:Ric(v)} and the scalar curvature \eqref{EQ:scal} are again positive. We then have $f(t)=N\sin(\frac{t-t'}{N})$ for $t\leq t_1'$ near $t_1'$ and $f'(t_1')=\cos(r_2)$.

        Next, we modify $h$ near $t_1'$ by choosing $\varepsilon>0$ and a smooth cutoff function $\psi\colon\R\to\R$ that is constant $1$ on $(-\infty,t_1'-\varepsilon]$ and constant $0$ on $[t_1',\infty)$ with $\psi'\leq 0$ globally and $\psi>0$ on $(-\infty,t_1')$. We then replace $h$ on $[t_1'-\varepsilon,t_1']$ by the unique function $\bar{h}$ that satisfies
        \begin{align}
            \bar{h}(t_1'-\varepsilon)&=h(t_1'-\varepsilon),\\
            \bar{h}'&=\psi h'.
        \end{align}
        Since all derivatives of $\psi$ vanish at $t=t_1'-\varepsilon$, replacing $h$ by $\bar{h}$ on $[t_1'-\varepsilon,t_1']$ again results in a smooth function, whose derivatives all vanish at $t=t_1'$. Moreover, we have
        \begin{equation}
            \bar{h}''=\psi' h+\psi h''\leq \psi h''
        \end{equation}
        and hence
        \begin{equation}
            -\frac{\bar{h}''}{\bar{h}'}\geq -\frac{h''}{h'}.
        \end{equation}
        This implies that the Ricci curvature \eqref{EQ:Ric(s)} is positive. Further, the Ricci curvatures \eqref{EQ:Ric(t)} and \eqref{EQ:Ric(v)} are also positive for $\varepsilon$ sufficiently small since $\bar{h}(t_1')\to h(t_1')$ as $\varepsilon\to 0$ and
        \begin{equation}
            \bar{h}'\leq \bar{h}'(t_1'-\varepsilon)=h'(t_1'-\varepsilon)\to h_1'(t_1')
        \end{equation}
        as $\varepsilon\to 0$. Also the scalar curvature \eqref{EQ:scal} remains positive since it is independent of $h$.
        
       % modify $f$ to be of the form $f(t)=N\sin(\frac{t-t'}{N})$ for some $N,t'>0$ near $t=t_1$, and so that all derivatives of $h$ vanish at $t=t_1$ while preserving the positivity of the Ricci curvature of the metric \eqref{EQ:doubly_warped}. The linear dependence of the scalar curvature of the induced metric on the second derivatives of $f$, as well as the fact that it is a linear combination of the terms $-\frac{f''}{f}$ and $\frac{1-{f'}^2}{f^2}$ ensure in the same way that it remains positive throughout the construction.
        
        Finally, the boundary conditions (i) and (ii) are achieved by appropriately rescaling the metric, i.e.\ replacing $h$ and $f$ by $\frac{1}{N} h(\cdot N)$ and $\frac{1}{N} f(\cdot N)$, and shifting their domain by $t'$. Note that a smaller value of $h(t_1)=r_1$ can always be achieved by choosing a larger value of $t_1$, so we can achieve condition (i) for all $r_1$ sufficiently small.
	\end{proof}
	
	Hence, we obtain the following result:
	\begin{proposition}\label{prop:Riceps}
		Let $(M^n,g)$ be a closed Riemannian manifold of dimension $n\geq 3$ with positive Ricci curvature. Then, for all $\ell\in\N$, there exists a metric $g_\ell$ of positive scalar curvature on $M\#(-M)\#_\ell(S^1\times S^{n-1})$ such that for all $\varepsilon>0$ there is a metric $\bar{g}^{\varepsilon}_\ell$ of positive Ricci curvature on $\Sigma_0 M\#_{\ell}(S^2\times S^{n-1})$ and a totally geodesic embedding
		\begin{equation} 
        (M\#(-M)\#_\ell(S^1\times S^{n-1}),g_\ell)\subseteq(\Sigma_0M\#_\ell(S^2\times S^{n-1}),\bar{g}^{\varepsilon}_\ell)
        \end{equation}
		with normal Ricci curvatures $\Ric^{\bar{g}^{\varepsilon}_\ell}(\nu)=\varepsilon$, where $\nu$ is a unit normal vector field of $M\#(-M)\#_\ell(S^1\times S^{n-1})$.
	\end{proposition}
	\begin{proof}
		We use the construction described above, i.e.\ we perform $(\ell+1)$ surgeries on $S^1\times M$ equipped with the product metric $r_1^2ds_1^2+g$ for some $r_1>0$, which has non-negative Ricci curvature. Note that the induced metric on the embedded totally geodesic hypersurface $S^0\times M$ is given by $g$, which has positive Ricci curvature, and hence also positive scalar curvature.
		
		By Lemma \ref{L:metric_tot_geod}, by choosing $r_1$ sufficiently small, we can now perform $(\ell+1)$ surgeries on this space, i.e.\ remove $(\ell+1)$ pairwise disjoint copies of $S^1\times D_{r_2}^{n}$ for $r_2>0$ sufficiently small, and smoothly replace each one with a copy of $D^2\times S^{n-1}$, equipped with a metric of strictly positive Ricci curvature. In addition, the scalar curvature of the induced metric on the hypersurface $D^1\times S^{n-1}\subseteq D^2\times S^{n-1}$ is strictly positive.
		
		Hence, we overall obtain a metric of non-negative Ricci curvature on $\Sigma_0 M\#_{\ell}(S^2\times S^{n-1})$ with points of strictly positive Ricci curvature, and a totally geodesic embedding of $M\#(-M)\#_{\ell}(S^1\times S^{n-1})$ whose induced metric has positive scalar curvature. We now slightly deform the metric on $\Sigma_0 M\#_{\ell}(S^2\times S^{n-1})$ to have globally strictly positive Ricci curvature while preserving the isometry group of the initial metric. The existence of such a deformation is sketched in \cite[p. 20]{Eh76} and a detailed proof is given in Theorem \ref{T:Ric>0_equiv} below.
        
        %Here we apply the deformation results of Ehrlich \cite[Theorem 5.1]{Eh76}, which show that such a deformation is possible $C^4$-close to the original metric.
		
		%As explained in \cite[p. 20]{Eh76}, this deformation can be arranged to preserve the isometry group of the original metric.
        In particular, since the embedding 
        \begin{equation}
            M\#(-M)\#_{\ell}(S^1\times S^{n-1})\subseteq \Sigma_0 M\#_{\ell}(S^2\times S^{n-1})
        \end{equation} is the fixed point set of the isometric involution induced be a reflection on $S^{1}$, it remains the fixed point set of this involution along the deformation. As a consequence, it remains a totally geodesic hypersurface. In addition, since the deformation is arbitrarily close to the initial metric, the induced metric on $M\#(-M)\#_{\ell}(S^1\times S^{n-1})$ again has positive scalar curvature.
		
		Finally, we apply the deformation of Proposition \ref{P:deformation} to obtain the required metric.
	\end{proof}

   \begin{proof}[End of proof of Theorem \ref{mainthm}]
       For a fixed $\ell\in \N$, let $g_\ell$ be the Riemannian metric on $N_{\ell}:=M\#(-M)\#_\ell(S^1\times S^{n-1})$ obtained from the application of Proposition \ref{prop:Riceps}. The Laplacian $\Delta_{g_\ell}$ has discrete eigenvalues 
       \begin{equation}\label{eigen}
                 0=\lambda_0<\lambda_1\leq\lambda_2\leq \cdots . \end{equation}
For any $0<\varepsilon<\lambda_{1}$, Proposition \eqref{prop:Riceps} also gives an embedding
    	\begin{equation} 
        (M\#(-M)\#_\ell(S^1\times S^{n-1}),g_\ell)\subseteq(\Sigma_0M\#_\ell(S^2\times S^{n-1}),\bar{g}^{\varepsilon}_\ell)
        \end{equation}
    that is totally geodesic and hence minimal. Since $\Ric^{\bar{g}^{\varepsilon}_\ell}(\nu) =\varepsilon$, the Jacobian operator of $N_\ell$ reduces to (c.f.  \cite[Definition 1.31]{ColdingMinicozzibook}): 
   $$ \textrm{L}_{N_\ell}=\Delta_{g_\ell} + \Ric^{\bar{g}^{\varepsilon}_\ell}(\nu) + |\II_{N_\ell}|^2=  \Delta_{g_\ell} + \varepsilon.$$    
 Thus, the Morse index of $\textrm{L}_{N_\ell}$ is equal to the number of eigenvalues of $\Delta_{g_\ell}$ less than $\varepsilon$. By \eqref{eigen}, that number must be 1.
\end{proof}

\begin{remark}
    Since in the above proof we can choose $\varepsilon$ arbitrary, and since the sequence of eigenvalues of $\Delta_{g_\ell}$ diverges, we can also realise arbitrarily high indices with the same argument (however, depending on the multiplicities of the eigenvalues, not every index can be realized).
\end{remark}

\appendix

\section{From non-negative to positive Ricci curvature in the equivariant setting}\label{appendix}

The purpose of this section is to give a detailed proof of the following theorem.

\begin{theorem}\label{T:Ric>0_equiv}
    Let $(M,g)$ be a closed Riemannian manifold of non-negative Ricci curvature on which a compact Lie group $G$ acts via isometries. Assume that there exists a point in $M$ at which all Ricci curvatures of $g$ are positive. Then $M$ admits a $G$-invariant Riemannian metric of positive Ricci curvature.
\end{theorem}
Theorem \ref{T:Ric>0_equiv} is proven in \cite{Eh74,Eh76} in the non-equivariant setting (i.e.\ when $G$ is the trivial group) by using a sequence of conformal modifications of the metric on small metric balls. Further, it is sketched in \cite[p. 20]{Eh76}, and with slightly more details in \cite[pp. 61--62]{Eh74}, how averaging the deformations over the action of $G$ extends this proof to the equivariant setting. However, the formulae for the Ricci curvature of the metric obtained by averaging given in \cite[p. 62]{Eh74} are incorrect, see Remark \ref{R:formulae_incor} below, so it is not guaranteed that one obtains positive Ricci curvature in this way. Therefore, we will use a different approach and first establish the existence of a global deformation as follows.
\begin{theorem}\label{T:Ric>0_def}
    Let $(M,g)$ be a closed Riemannian manifold of non-negative Ricci curvature that contains a point of positive Ricci curvature. Then there exist a smooth family of metrics $g(t)$, $t\in[0,\varepsilon)$, on $M$ and $c>0$ such that $g(0)=g$ and
    \begin{equation}\label{EQ:Ric_der}
        \left.\frac{d}{dt}\right|_{t=0}\Ric^{g(t)}(v,v)>0
    \end{equation}
    for all $v\in TM$ with $g(v,v)=1$ and $\Ric^g(v,v)\leq c$. In particular, $g(t)$ has positive Ricci curvature for all $t>0$ sufficiently small.
\end{theorem}

The family $g(t)$ will be of the form $g(t)=e^{2tf}g$ for a function $f\colon M\to\R$. The Ricci curvatures of such a metric are given as follows.
\begin{lemma}[{\cite[p.\ 156]{Pe16}}]\label{L:conf_change}
    Let $(M^n,g)$ be a Riemannian manifold, $f\colon M\to\R$ a smooth function and $t>0$. Then the Ricci tensor of the metric $g(t)=e^{2tf}g$ is given by
    \begin{equation*}
        \Ric^{g(t)}=\Ric^g-t\left( (n-2)\mathrm{Hess}_f+(\Delta f) g\right)+t^2(n-2)\left( (df)^2-\lVert \nabla f\rVert^2_g g \right).
    \end{equation*}
    Here, the Hessian $\mathrm{Hess}_f$ and Laplacian $\Delta f$ are taken w.r.t.\ the metric $g$ and we use the sign conventions $\mathrm{Hess}_f(u,v)=g(\nabla_u\nabla f,v)$ and $\Delta f=\mathrm{tr}\mathrm{Hess}_f$.
\end{lemma}
It follows that for the family $g(t)$ in Lemma \ref{L:conf_change}, we have
\begin{equation}\label{EQ:conf_change_der}
    \left.\frac{d}{dt}\right|_{t=0}\Ric^{g(t)}=-\left( (n-2)\mathrm{Hess}_f+(\Delta f) g\right).
\end{equation}

%We now recall some results from \cite{Eh76}. First, given a closed manifold $M$, we denote by $\mathcal{R}(M)$ the space of all $C^2$-Riemannian metrics on $M$. The $C^2$-norm $\lVert g\rVert_{C^2}$ of a metric $g\in \mathcal{R}(M)$ is then defined as the maximum of absolute values of the coefficients of $g$ and their first and second derivatives in fixed coordinate charts of $M$.

Now let $g$ be a Riemannian metric on $M$ and let $x\in M$. Then the distance function $r_x(y)=d_g(x,y)$ is smooth outside of $x$ and the cut locus of $x$. For $R>0$ we define the function $\rho_{x,R}=R-r_x$ on the metric ball $B_R(x)$. If $R$ is smaller than the injectivity radius of $g$, then $\rho_{x,R}$ is a smooth function on $B_R(x)\setminus\{x\}$. If we extend $\rho_{x,R}$ constantly by $0$ outside of $B_R(x)$, then we obtain a continuous function on $M$. Moreover, for any $k\in\N$, the function $\rho^k_{x,R}$ is of class $C^{k-1}$ on $M\setminus\{x\}$. Finally, to obtain a function that has suitable differentiability conditions on all of $M$, we smooth $\rho_{x,R}$ in a small neighbourhood of $x$.

We set
\begin{equation}
    f_{x,R}(y)=-\rho_{x,R}(y)^5\, .
\end{equation}
Then $f_{x,R}$ is of class $C^4$. Since the analysis of the Hessian $\mathrm{Hess}_{f_{x,R}}$ involves the Hessian of the distance function $r_x$, it is convenient to restrict to metric balls that are \emph{convex}, that is, we choose $R(g)>0$ such that for all $x\in M$
\begin{equation}\label{EQ:dist_Hess}
    \left(2-\tfrac{1}{4}\right)g\leq \mathrm{Hess}_{r_x^2}\leq \left(2+\tfrac{1}{4}\right)g
\end{equation}
holds at all points $y\in B_{R(g)}(x)$. Note that $R(g)$ exists by compactness of $M$ and the fact that the Hessian $\mathrm{Hess}_{r_x^2}$ is given by $2g$ at $x$ (see also \cite[Theorem 3.5]{Eh76}).

% The result in this context we will use is the following.
% \begin{proposition}[{\cite[Theorem 3.5]{Eh76}}]\label{P:convex}
%     Let $(M,g_0)$ be a closed Riemannian manifold. Then there are constants $\delta(g_0)>0$ and $R(g_0)>0$ such that for any metric $g\in\mathcal{R}(M)$ with $\lVert g-g_0\rVert_{C^2}<\delta(g_0)$ for any $x\in M$ any $g$-metric ball around $x$ of radius at most $R(g_0)$ is $g$-convex and
%     \begin{equation}\label{EQ:dist_Hess}
%         \left(2-\tfrac{1}{4}\right)g\leq \mathrm{Hess}_{r_x^2}\leq \left(2+\tfrac{1}{4}\right)g.
%     \end{equation}
% \end{proposition}
This allows us to prove the following.
\begin{proposition}\label{P:Hessian_est}
    There are constants $\delta=\delta(n)\in(0,\frac{1}{3})$ and $\varepsilon=\varepsilon(n)>0$ that only depend on $n$ such that the following holds. Let $(M^n,g)$ be a closed Riemannian manifold and consider the constant $R(g)$ obtained in \eqref{EQ:dist_Hess}. Then for all $x\in M$ and $R\in(0,R(g)]$ the function $f=f_{x,R}$ satisfies the following:
    \begin{align}
        -\left( (n-2)\mathrm{Hess}_{f}(v,v)+(\Delta f) g(v,v) \right)\geq0\quad &\text{for }r_x\in[(1-3\delta) R,R],\label{EQ:delta_est1}\\
        -\left( (n-2)\mathrm{Hess}_{f}(v,v)+(\Delta f) g(v,v) \right)\geq\varepsilon R^3\quad &\text{for }r_x\in[(1-3\delta) R,(1-\delta)R],\label{EQ:delta_est2}
    \end{align}
    where $v$ is a unit tangent vector.
\end{proposition}
\begin{proof}
    We follow the calculations in \cite[p. 16]{Eh76}. We first consider the Hessian $\mathrm{Hess}_{f}$.
    \begin{align}
        -\mathrm{Hess}_f(v,v)&=\mathrm{Hess}_{\rho_{x,R}^5}(v,v)\notag\\
        &=g(\nabla_v\nabla(\rho_{x,R}^5),v)\notag\\
        &=5v(\rho_{x,R}^4)g(\nabla \rho_{x,R},v)+5\rho_{x,R}^4g(\nabla_v\nabla\rho_{x,R},v)\notag\\
        &=20v(\rho_{x,R})^2\rho_{x,R}^3+5\rho_{x,R}^4\mathrm{Hess}_{\rho_{x,R}}(v,v)\notag\\
        &\geq 5\rho_{x,R}^4\mathrm{Hess}_{\rho_{x,R}}(v,v).
    \end{align}
    Hence, for a local orthonormal frame $(e_1,\dots,e_n)$ with $e_1=\nabla r_x$, we have
    \begin{align}
        -\Delta f&=-\sum_i \mathrm{Hess}_{f}(e_i,e_i)\notag\\
        &=20e_1(\rho_{x,R})^2\rho_{x,R}^3+5\rho_{x,R}^4\Delta\rho_{x,R}\notag\\
        &=20\rho_{x,R}^3+5\rho_{x,R}^4\Delta \rho_{x,R}.
    \end{align}
    
    Since $\rho_{x,R}=R-r_x$ on a small neighbourhood of $x$, it therefore remains to consider $\mathrm{Hess}_{r_x}$. We have
    \begin{equation}
        \mathrm{Hess}_{r_x^2}=2r_x\mathrm{Hess}_{r_x}+2(dr_x)^2.
    \end{equation}
    By \eqref{EQ:dist_Hess}, we therefore obtain
    \begin{equation}
        \frac{1-\frac{1}{8}}{r_x}g-\frac{(dr_x)^2}{r_x}\leq \mathrm{Hess}_{r_x}\leq \frac{1+\frac{1}{8}}{r_x}g-\frac{(dr_x)^2}{r_x}.
    \end{equation}
    Hence,
    \begin{equation*}
        |\mathrm{Hess}_{\rho_{x,R}}(v,v)|=|\mathrm{Hess}_{r_x}(v,v)|\leq \frac{9}{8(R-\rho_{x,R})}g(v,v)=\frac{9}{8(R-\rho_{x,R})}
    \end{equation*}
    and, since $\mathrm{Hess}_{r_x}(\nabla r_x,\nabla r_x)=0$,
    \begin{equation}
        |\Delta \rho_{x,R}|\leq |\mathrm{Hess}_{r_x}(e_1,e_1)|+\sum_{i=2}^n|\mathrm{Hess}_{r_x}(e_i,e_i)|\leq \frac{9(n-1)}{8(R-\rho_{x,R})}.
    \end{equation}

    Altogether we therefore obtain
    \begin{align}
        -((n-2)\mathrm{Hess}_f(v,v)&+(\Delta f)g(v,v))\notag\\
        &\geq 5(n-2)\rho_{x,R}^4\mathrm{Hess}_{\rho_{x,R}}+5\rho_{x,R}^4\Delta \rho_{x,R}+20 \rho_{x,R}^3\notag\\
        &\geq 20\rho_{x,R}^3-5\rho_{x,R}^4\frac{9(2n-3)}{8(R-\rho_{x,R})}\notag\\
        &=5\rho_{x,R}^3\left( 4-\rho_{x,R}\frac{9(2n-3)}{8(R-\rho_{x,R})} \right).
    \end{align}

    We set $\delta(n)=\frac{16}{3(9(2n-3)+16)}$, so that $\frac{9}{8}(2n-3)\frac{3\delta(n)}{1-3(\delta(n))}=2$. Then, for $r_x\in[(1-3\delta(n))R,R]$, i.e.\ $\rho_{x,R}\in[0,3\delta(n)R]$, we have
    \begin{align}
        4-\rho_{x,R}\frac{9(2n-3)}{8r_x}\geq 4-3\delta(n)R\frac{9(2n-3)}{8(1-3\delta(n))R}=2
    \end{align}
    and so $-((n-2)\mathrm{Hess}_f(v,v)+(\Delta f)g(v,v))\geq 0$.

    Finally, set $\varepsilon(n)=10\delta(n)^3$. Then for $r_x\in[(1-3\delta(n))R,(1-\delta(n))R]$, i.e.\ $\rho_{x,R}\in[\delta(n)R,3\delta(n)R]$, we have
    \begin{align}
        -((n-2)\mathrm{Hess}_f(v,v)&+(\Delta f)g(v,v))\notag\\
        &\geq 5\delta(n)^3R^3\left( 4-3\delta(n)R\frac{9(2n-3)}{8(1-3\delta(n))R} \right)\notag\\
        &=10\delta(n)^3R^3\notag\\
        &=\varepsilon(n)^3 R^3.
    \end{align}
\end{proof}

\begin{proof}[Proof of Theorem \ref{T:Ric>0_def}]
    Let $x_0\in M$ be a point at which all the Ricci curvatures of $g$ are positive. We choose $c>0$ so that the Ricci curvature at $x_0$ is at least $2c$. Let $R(g)$ be the constant from \eqref{EQ:dist_Hess} and let $\varepsilon=\varepsilon(n)$, $\delta=\delta(n)$ be the constants from Proposition \ref{P:Hessian_est}. We choose $R\in(0,R(g)]$ so small that all points in $B_{(1-3\delta)R}(x)$ have Ricci curvature at least $c$.

    We now cover $M$ with metric balls of radius $R$ as follows. We set $x^0_{0}=x_0$. Since the union
    \begin{equation}
        \bigcup_{x\in B_{2\delta R}(x_0)}B_{(1-\delta)R}(x)
    \end{equation}
    covers all of $\overline{B}_R(x_0)\setminus B_{(1-3\delta)R}(x)$, by compactness, we can choose finitely many points $x_1^1,\dots,x_1^{k_1}\in B_{2\delta R}(x_0)$ such that the union\linebreak $\bigcup_{i} B_{(1-\delta)R}(x_1^i)$ covers all of $\overline{B}_{R}(x_0)\setminus B_{(1-3\delta)R}(x)$. Note that the balls $B_{(1-3\delta)R}(x_1^i)$ lie entirely in $B_{(1-\delta)}(x_0)$.

    More generally, set $B_j=B_{(1+(j-1)\delta)R}(x_0)$ and assume we have covered the ball $\overline{B}_j$. We then choose points
    \begin{equation}
        x_{j+1}^1,\dots,x_{j+1}^{k_{j+1}}\in B_{\delta(j+2)R}(x_0)
    \end{equation}
    (so the balls $B_{(1-3\delta)R}(x_j^i)$ lie entirely in $B_j$) such that the union\linebreak $\bigcup_{i} B_{(1-\delta)R}(x_{j+1}^i)$ covers all of $\overline{B}_{j+1}\setminus B_j$.
    
    Since $M$ is compact, there exists $\ell\in\N$ such that $B_j=M$ for all $j\geq \ell$.
    % \begin{equation}
    %     B_{(1-3\delta)R}(x_0)\cup\bigcup_{j=0}^\ell\bigcup_{i=1}^{k_j}\left(\overline{B}_{(1-\delta)R}(x_j^i)\setminus B_{(1-3\delta)R}(x_j^i)\right)=M.
    % \end{equation}
    %We note that this sequence of points $\{x_j^i\}$ is chosen so that the following property holds. For all $x\in M$, either $x\in B_{(1-3\delta)R}(x_0)$, or there exists a point $x_j^i$ such that $d(x_j^i,x)\in[(1-3\delta)R,(1-\delta)R)$ and we have $d(x_{j}^{i'},x)\geq (1-3\delta)R$ for all $i'\in\{1,\dots,k_j\}$ and $d(x_{j'}^{i'},x)\geq (1-\delta)R$ for all $j'<j$ and $i'\in\{1,\dots,k_{j'}\}$.
    For $\lambda_0,\dots,\lambda_\ell>0$ we then set
    \begin{equation}
        f=\sum_{j=0}^\ell\lambda_j\sum_{i=1}^{k_j}f_{x_{j}^i,R}.
    \end{equation}
    We will show that we can choose $\lambda_0,\dots,\lambda_\ell$ such that $-((n-2)\mathrm{Hess}_f+(\Delta f)g)$ is positive on $M\setminus B_{(1-3\delta)R}(x_0)$. Since all points in $B_{(1-3\delta)R}(x_0)$ have Ricci curvatures at least $c$, together with \eqref{EQ:conf_change_der}, this shows that \eqref{EQ:Ric_der} holds for the metric $g(t)=e^{2tf}g$.

    First, consider points $x\in \overline{B}_0\setminus B_{(1-3\delta)R}(x_0)$, i.e.\ $x\in \overline{B}_{(1-\delta)R}(x_0)\setminus B_{(1-3\delta)R}(x_0)$. By Proposition \ref{P:Hessian_est}, for a unit vector $v\in T_xM$, we have
    \begin{align*}
        -((n-2)&\mathrm{Hess}_f(v,v)+(\Delta f)g(v,v))\notag\\
        =&-\lambda_0((n-2)\mathrm{Hess}_{f_{x_0,R}}(v,v)+(\Delta f_{x_0,R})g(v,v))\notag\\
        &-\sum_{j=1}^\ell\lambda_j\sum_{i=1}^{k_j}((n-2)\mathrm{Hess}_{f_{x_j^i,R}}(v,v)+(\Delta f_{x_j^i,R})g(v,v))\notag\\
        \geq &\lambda_0\varepsilon R^3-\sum_{j=1}^\ell \lambda_j\sum_{i=1}^{k_j}\left|(n-2)\mathrm{Hess}_{f_{x_j^i,R}}(v,v)+(\Delta f_{x_j^i,R})g(v,v)\right|.
    \end{align*}
    This expression is positive for $\lambda_0=1$ and for all $\lambda_j$ with $j\in\{1,\dots,\ell\}$ sufficiently small since each term 
    \begin{equation}
    |(n-2)\mathrm{Hess}_{f_{x_j^i,R}}(v,v)+(\Delta f_{x_j^i,R})g(v,v)|
    \end{equation}
    is uniformly bounded by the compactness of $M$.
    
    Next, consider $x\in \overline{B}_{j_0+1}\setminus B_{j_0}$ and assume $\lambda_0,\dots,\lambda_{j_0}$ are fixed. Then, by construction, there exists $i_0\in \{1,\dots,k_{j_0+1}\}$ such that
    \begin{equation}
        d(x_{j_0+1}^{i_0},x)\in[(1-3\delta)R,(1-\delta)R].
    \end{equation}
    Further, we have $d(x_{j_0+1}^i,x)\geq (1-3\delta)R$ for all $i\in\{1,\dots,k_{j_0+1}\}$ and $d(x_j^i,x)\geq (1-\delta)R$ for all $j<j_0+1$ and $i\in\{1,\dots,k_j\}$. Hence, by \eqref{EQ:delta_est1} and \eqref{EQ:delta_est2}, we have
    \begin{align*}
        -((n-2)&\mathrm{Hess}_f(v,v)+(\Delta f)g(v,v))\notag\\
        =& -\sum_{j=0}^{j_0}\lambda_j\sum_{i=1}^{k_j}((n-2)\mathrm{Hess}_{f_{x_j^i,R}}(v,v)+(\Delta f_{x_j^i,R})g(v,v))\notag\\
        &-\lambda_{j_0+1}((n-2)\mathrm{Hess}_{f_{x_{j_0+1}^{i_0},R}}(v,v)+(\Delta f_{x_{j_0+1}^{i_0},R})g(v,v))\notag\\
        &-\lambda_{j_0+1}\sum_{i\neq i_0}((n-2)\mathrm{Hess}_{f_{x_{j_0+1}^i,R}}(v,v)+(\Delta f_{x_{j_0+1}^i,R})g(v,v))\notag\\
        &-\sum_{j=j_0+2}^\ell\lambda_j\sum_{i=1}^{k_j}((n-2)\mathrm{Hess}_{f_{x_j^i,R}}(v,v)+(\Delta f_{x_j^i,R})g(v,v))\notag\\
        \geq & \lambda_{j_0+1}\varepsilon R^3\notag\\
        &-\sum_{j=j_0+2}^\ell \lambda_j\sum_{i=1}^{k_j}\left|(n-2)\mathrm{Hess}_{f_{x_j^i,R}}(v,v)+(\Delta f_{x_j^i,R})g(v,v)\right|.
    \end{align*}
    Again by bounding the terms 
    \begin{equation}
    |(n-2)\mathrm{Hess}_{f_{x_j^i,R}}(v,v)+(\Delta f_{x_j^i,R})g(v,v)| 
    \end{equation}
    uniformly from above, we obtain that this expression is positive for all $\lambda_{j_0+2},\dots,\lambda_\ell$ sufficiently small.
    This shows that \eqref{EQ:Ric_der} holds for $g(t)=e^{2tf}g$. Since $f$ is merely $C^4$, we apply a smoothing to $f$. Since \eqref{EQ:Ric_der} is an open condition, it is preserved for smoothings of $f$ that are sufficiently $C^2$-close.

    Finally, since the set of unit tangent vectors $v\in TM$ with $\mathrm{Ric}^g(v,v)\leq c$ is compact, there exists $t_0>0$ such that $g(t)$ has positive Ricci curvature for all $t\in(0,t_0]$.
\end{proof}

\begin{proposition}\label{P:averaging}
    Let $M$ be a closed Riemannian manifold on which a compact Lie group $G$ acts. Let $g(t)$, $t\in[0,\varepsilon)$ be a smooth family of Riemannian metrics on $M$ such that $g=g(0)$ is $G$-invariant. Define the family of metrics
    \begin{equation}
        \tilde{g}(t)=\int_G\varphi^*g(t)d\mu(\varphi),
    \end{equation}
    where $d\mu$ is the normalized Haar measure of $G$. Then
    \begin{equation}\label{EQ:Ric_av}
        \left.\frac{d}{dt}\right|_{t=0}\Ric^{\tilde{g}(t)}=\int_G\varphi^*\left(\left.\frac{d}{dt}\right|_{t=0}\Ric^{g(t)}\right)d\mu(\varphi). 
    \end{equation}
\end{proposition}
\begin{proof}
    We follow Weinstein's argument \cite{We70} for the sectional curvature. For that, recall that for a vector bundle $E\to M$, the space of smooth sections $C^\infty(M,E)$ together with the uniform $C^k$-metrics on $M$ with respect to a fixed background metric on $M$ and a covariant derivative on $E$ is a Fréchet manifold. Similarly, the space of Riemannian metrics $\mathcal{R}(M)$ on $M$ is a Fréchet manifold by viewing it as an open cone in $C^\infty(M,S^2 T^*M)$, where $S^2 T^*M$ is the bundle of symmetric $(0,2)$-tensors on $M$ (see e.g.\ \cite[Chapter 1]{TW15} for more details).

    Then the Ricci tensor can be viewed as a smooth map
    \begin{equation}
        \Ric\colon\mathcal{R}(M)\to C^\infty(M,S^2 T^*M),
    \end{equation}
    between Fréchet manifolds. Hence, for every $g\in\mathcal{R}(M)$, we can consider its differential
    \begin{equation}
        D_g\Ric\colon C^\infty(M,S^2 T^*M)\to C^\infty(M,S^2 T^*M),
    \end{equation}
    which is a linear differential operator of order 2. Hence, we obtain the following:
    \begin{align}
        \left.\frac{d}{dt}\right|_{t=0}\Ric^{\tilde{g}(t)}&=D_{g}\Ric\left( \left.\frac{d}{dt}\right|_{t=0}\tilde{g}(t)\right)\notag\\
        &=D_g\Ric\int_G\varphi^*\left( \left.\frac{d}{dt}\right|_{t=0}g(t)\right)d\mu(\varphi)\notag\\
        &=\int_G D_g\Ric\left(\varphi^*\left( \left.\frac{d}{dt}\right|_{t=0}g(t)\right)\right)d\mu(\varphi)\label{EQ:int_swap}\\
        &=\int_G\varphi^*\left( D_g\Ric\left( \left.\frac{d}{dt}\right|_{t=0}g(t)\right)  \right)d\mu(\varphi)\label{EQ:diff_inv}\\
        &=\int_G\varphi^*\left( \left.\frac{d}{dt}\right|_{t=0}\Ric^{g(t)}\right)d\mu(\varphi).\notag
    \end{align}
    To obtain \eqref{EQ:int_swap}, we used that $D_g\Ric$ is a linear differential operator, and to obtain \eqref{EQ:diff_inv} we used that $\Ric$ is equivariant with respect to the pull-back action of the diffeomorphism group and hence
    \begin{equation}
        D_{\varphi^*g}\Ric(\varphi^*T)=\varphi^*(D_g\Ric(T))
    \end{equation}
    for any diffeomorphism $\varphi\colon M\to M$ and any $T\in C^\infty(M,S^2 T^*M)$. Hence, the assumption $\varphi^*g=g$ for all $\varphi\in G$ yields \eqref{EQ:diff_inv}.
\end{proof}
\begin{remark}\label{R:formulae_incor}
    In \cite[p.\ 62]{Eh74} it is claimed that an equation similar to \eqref{EQ:Ric_av} holds for all values of $t$ and hence integrating provides an analogous equation for the Ricci tensor itself. However, since it is crucial for \eqref{EQ:diff_inv} that $g(0)$ is $G$-invariant, these formulae cannot be derived in this way.
    
    In fact, since the space of metrics on $M$ is path-connected, this would imply that the averaging construction of Proposition \ref{P:averaging} commutes with the Ricci tensor. However, this is not correct as can be seen, for example, by considering conformal changes of metrics as in Lemma \ref{L:conf_change} and using that the Ricci curvature is not linear in $e^{2tf}$.
\end{remark}

\begin{proof}[Proof of Theorem \ref{T:Ric>0_equiv}]
    Assume that $g$ is $G$-invariant. By Theorem \ref{T:Ric>0_def}, there is a smooth family of metrics $g(t)$, $t\in[0,\varepsilon)$ such that $\Ric^{g(t)}$ has positive derivative at $t=0$ on unit vectors with Ricci curvature at most $c$ for some $c>0$.

    We consider the averaged family $\tilde{g}(t)$ obtained from $g(t)$ by averaging as in Proposition \ref{P:averaging}, which is $G$-invariant by construction. Then, since $g$ is $G$-invariant, the set
    \begin{equation}
        R_c=\{v\in TM\mid g(v,v)=1\text{ and }\Ric^g(v,v)\leq c\}
    \end{equation}
    is $G$-invariant as well. Hence, by \eqref{EQ:Ric_av}, we obtain
    \begin{equation}
        \left.\frac{d}{dt}\right|_{t=0}\Ric^{\tilde{g}(t)}(v,v)=\int_G\left.\frac{d}{dt}\right|_{t=0}\Ric^{g(t)}(\varphi_*v,\varphi_* v)d\mu(\varphi)>0 
    \end{equation}
    for all $v\in R_c$, so \eqref{EQ:Ric_der} holds for the family $\tilde{g}(t)$ as well. It follows that $\tilde{g}(t)$ has positive Ricci curvature for all $t$ sufficiently small as in Theorem \ref{T:Ric>0_def}.
\end{proof}

\section*{Declarations} 
On behalf of all authors, the corresponding author states that there is no conflict of interest or associated data in our manuscript.

\bibliography{bib} 
\bibliographystyle{amsalpha}
	
\end{document}